\documentclass[12pt,a4paper]{article}

\usepackage[T1]{fontenc}
\usepackage{authblk}

\usepackage[utf8]{inputenc}
\usepackage[english]{babel}
\usepackage{amsmath}
\usepackage{amsfonts}
\usepackage{amssymb}
\usepackage{amsthm}
\usepackage{mathtools}
\usepackage[left=2cm,right=2cm,top=2cm,bottom=2cm]{geometry}
\usepackage{algorithm}
\usepackage{algpseudocode}
\usepackage{subfigure}
\usepackage{graphicx}
\usepackage{color}
\usepackage{caption}
\usepackage{subfig}
\usepackage{verbatim}

\newtheorem{thm}{Theorem}[section]
\newtheorem{definition}[thm]{Definition}

\newtheorem{prop}[thm]{Proposition}

\newtheorem{rem}[thm]{Remark}

\makeatletter
\def\BState{\State\hskip-\ALG@thistlm}
\makeatother
\author[1]{Bernadett \'Acs} 
\author[1]{Gergely Szlobodnyik}
\author[1,2]{G\'abor Szederk\'enyi\footnote{corresponding author}}
\affil[1]{\small P\'azm\'any P\'eter Catholic University, Faculty of Information Technology and Bionics, Práter u. 50/a,  H-1083 Budapest, Hungary} 
\affil[2]{\small Process Control Research Group, Systems and Control Laboratory, Institute for Computer Science and Control (MTA SZTAKI) of the Hungarian Academy of Sciences, Kende u. 13-17, H-1111 Budapest, Hungary}
\affil[ ]{\small e-mail: szederkenyi@itk.ppke.hu}
\title{A computational approach to the structural analysis of uncertain kinetic systems}
\date{}

\begin{document}

\maketitle

\begin{abstract}
A computation-oriented representation of uncertain kinetic systems is introduced and analysed in this paper. It is assumed that the monomial coefficients of the ODEs belong to a polytopic set, which defines a set of dynamical systems for an uncertain model. An optimization-based computation model is proposed for the structural analysis of uncertain models. It is shown that the so-called dense realization containing the maximum number of reactions (directed edges) is computable in polynomial time, and it forms a super-structure among all the possible reaction graphs corresponding to an uncertain kinetic model, assuming a fixed set of complexes. The set of core reactions present in all reaction graphs of an uncertain model is also studied. Most importantly, an algorithm is proposed to compute all possible reaction graph structures for an uncertain kinetic model.
\end{abstract}


\section{Introduction}
Kinetic models in the form of nonlinear ordinary differential equations are widely used for describing time-varying physico-chemical quantities in (bio-)chemical environments \cite{Erdi1989}. Moreover, the kinetic system class is dynamically rich enough to characterize general nonlinear behaviour in other application fields as well, particularly where the state variables are nonnegative and the model has a networked structure, such as in the modelling of process systems, population or disease dynamics, or even transportation processes \cite{Haddad2010,Chellaboina2009,Tak:96}. In biochemical applications, the exact values (or even sharp estimates) of the model parameters are often not known, making the models uncertain \cite{Chen2012}. Even when we have measurements of sufficient quantity and quality, the lack of structural or practical identifiability may result in highly uncertain models even with the most sophisticated estimation methods \cite{Craciun2008,Szederkenyi2011c,Chis2016,Chis2011}. This inherent uncertainty was a key factor in the development of Chemical Reaction Network Theory (CRNT), where (among other goals) a primary interest is to study the relations between the network structure and the qualitative properties of the corresponding dynamics, preferably without the precise knowledge of model parameters. From the earlier results of CRNT, we have to mention the well-known Deficiency One and Deficiency
Zero Theorems \cite{Feinberg1987,Son:2001} opening the way towards a structure-based (essentially parameter-free) dynamical analysis of biological networks. Recent particularly important findings in this area are the identification of biologically plausible structural sources of absolute concentration robustness \cite{Shinar2010},  and the proof of the Global Attractor Conjecture \cite{Anderson2011,Craciun2015}.

The efficient treatment of uncertain quantitative models is a fundamental task in mathematics, physics, (bio)chemistry and in related engineering fields \cite{Briggs2016,Guy2015}.  An important early result is \cite{Harrison1979}, where the solutions of linear compartmental systems are studied with uncertain flow rates that are assumed to belong to known intervals. In \cite{Liebermeister2005} a probabilistic framework is proposed for the representation and analysis of uncertain kinetic systems. %
In \cite{Nagy2011} an analytical expression is computed for the temperature dependence of the uncertainty of reaction rate coefficients, and a method is proposed for computing the covariance matrix and the joint probability density function of the Arrhenius parameters.
A recent outstanding result is \cite{Schillings2015}, where a deterministic computation interpolation scheme for uncertain reaction network models is proposed, which is able to handle large-scale models with hundreds of species and kinetic parameters.

The description of model uncertainties using convex sets is often a computationally appealing way of solving model analysis, estimation or control problems \cite{Weinmann1991,Boyd1994}. From the numerous applications, we mention here only a few selected works from different fields. In \cite{Wu2006}, a stabilization scheme was given for nonlinear control system models, where the uncertain coefficients of smooth basis functions in the system equations are assumed to form a polytopic set. An interval representation of fluxes in metabolic networks was introduced in \cite{Llaneras2007}, which enables the computation of the $\alpha$-spectrum even from an uncertain flux distribution. In \cite{Liptak2016}, a nonlinear feedback design method is proposed which is able to robustly stabilize parametrically uncertain kinetic systems using the convexity of the constraint ensuring the complex balance property. Recently, a new approach was given for the stability analysis of general Lotka-Volterra models with polytopic parameter uncertainties in \cite{Badri2017}. 

It is known from the fundamental dogma of chemical kinetics that the reaction graph structure corresponding to a kinetic ODE-model is generally non-unique, even in the case when the rate coefficients are assumed to be known \cite{Horn1972,Erdi1989,Schnell2006}. This property is usually called dynamical equivalence, macro-equivalence or confoundability in the literature \cite{Craciun2008,Horn1972}.  The first solution to the inverse problem, namely the construction of one possible reaction network (called the canonical network) for a given set of kinetic differential equations was described in \cite{Hars1981}. The notion of dynamical equivalence was extended by introducing linear conjugacy of kinetic systems in \cite{Johnston2011conj} allowing a positive diagonal transformation between the solutions of the kinetic differential equations. The simple factorization of kinetic models containing the Laplacian matrix of the reaction graph allows the development of 
efficient methods in various optimization frameworks for computing reaction networks realizing or linearly conjugate to a given dynamics with preferred properties such as density/sparsity \cite{Szederkenyi2009b}, weak reversibility \cite{Johnston2012}, complex or detailed balance \cite{Szederkenyi2011a}, minimal or zero deficiency  
\cite{Johnston2013,Liptak2015}. Using the superstructure property of the so-called dense realizations, it is possible to algorithmically generate all possible reaction graph structures corresponding to linearly conjugate realizations of a kinetic dynamics \cite{Acs2016,Acs2017}.

Even if the monomials of a kinetic system are known, the parameters (i.e., the monomial 
coefficients) are often uncertain in practice. For example, one may consider the situation when a kinetic polynomial 
ODE model with fixed structure is identified from noisy measurement data. In such a case, using the covariance matrix of the 
estimates and the nonnegativity/kinetic constraints for the system model, we can define a simple interval-based (see, e.g. \cite{Llaneras2007}), or more general (e.g., polytopic or ellipsoidal) uncertain model \cite{Ljung1999,Szederkenyi2016a}. Based on the above, the goal of this paper is to extend and illustrate previously introduced notions, computational models and algorithms for kinetic systems with polytopic uncertainty. 
\section{Notations and computational background}
In this section we summarize the basic notions of kinetic polynomial systems and the generalized model defined 
with uncertain parameters.

\bigskip

\noindent The applied general notations are listed below:

\begin{tabular}{ll}
$\mathbb{R}$ & the set of real numbers \\
$\mathbb{R_+}$ & the set of nonnegative real numbers \\
$\mathbb{N}$ & the set of natural numbers \\
$H^{n\times m}$ & the set of matrices having entries from a set $H$ with $n$ rows and $m$ columns\\
$[M]_{ij}$ & the entry of matrix $M$ with row index $i$ and column index $j$\\
$[M]_{.j}$ & the $j$th column of matrix $M$\\
$R_j$ & the $j$th coordinate of vector $R$\\
$\mathbf{0}^n$ & the null vector in $\mathbb{R}^n$\\
$\mathbf{1}^n$ & a vector in $\mathbb{R}^n$ with all coordinates equal to 1\\
$e_i^n$ & a vector in $\mathbb{R}^n$ for which the $i$th coordinate is 1 and all the others are zero
\end{tabular}

\subsection{Kinetic polynomial systems and their models}
Nonnegative polynomial systems are defined in the following general form:
\begin{equation} \label{eq:kinpoly}
\dot{x}=M \cdot \varphi(x)
\end{equation}
where $x: \mathbb{R} \rightarrow \mathbb{R}_{+}^{n}$ is a nonnegative valued function, $M \in \mathbb{R}^{n \times p}$ is a coefficient matrix and   
$\varphi: \mathbb{R}_+^n \rightarrow \mathbb{R}_+^p$ is a monomial-type 
vector-mapping. 
The invariance of the nonnegative orthant with respect to the dynamics \eqref{eq:kinpoly} can be ensured by 
prescribing sign conditions for the entries of matrix $M$ depending on the exponents of $\varphi$, 
see \cite{Erdi1989,Haddad2010}.

In this paper, we treat kinetic models as a general nonlinear system class that is suitable for the description of biochemical reaction networks. Hence, we do not require that all models belonging to the studied class are actually chemically realizable. Several physically or chemically relevant properties such as component mass conservation, detailed or complex balance 
can be ensured by adding further constraints to the computations (see, e.g. \cite{Johnston2012a}).

\begin{definition}
A \textbf{chemical reaction network} (CRN) can be characterized by three sets \emph{\cite{Feinberg:79, Feinberg1987}}.\\
\begin{tabular}{ll}
\textbf{species}: & $\mathcal{S}= \{X_i\  |\  i \in \{1,\ldots,n\}\}$ \\
\textbf{complexes}: & $\mathcal{C} = \{ C_j = \sum \limits_{i=1}^{n} \alpha_{ji} X_i\ |\ \alpha_{ji} \in \mathbb{N}, \  j\in\{1,\ldots,m\}, \  i \in \{1,\ldots,n\}\}$ \\
\textbf{reactions}: & $\mathcal{R} \subseteq \{(C_i,C_j)\  |\  C_i,C_j \in \mathcal{C}\}$  \\
\end{tabular}\\
For all $i,j \in \{1, \ldots m\}$, $i \neq j$ the reaction $C_i \rightarrow C_j$ is represented by the ordered 
pair  $(C_i,C_j)$, and it is described by a nonnegative real number $k_{ij} \in \mathbb{R}^+$ called 
\textbf{reaction rate coefficient}. 
The reaction $C_i \rightarrow C_j$ is present in the reaction network if and only if $k_{ij}$ is strictly 
positive.
\end{definition}

The relation between species and complexes is described by the \textbf{complex composition matrix} 
$Y \in \mathbb{R}^{n \times m}$, the columns of which correspond to the complexes, i.e. 
\begin{equation}
[Y]_{ij}=\alpha_{ji} \qquad \quad i\in \{1, \ldots, n\}, \  j \in \{1, \ldots, m\}
\end{equation}
The presence  of the reactions in the CRN is defined through the rate coefficients as the off-diagonal entries 
of the \textbf{Kirchhoff matrix} $A_k \in \mathbb{R}^{m \times m}$ which is a Metzler compartmental matrix with zero column-sums. Its entries are defined as:
\begin{equation}
[A_k]_{ij}= \begin{cases} 
k_{ji} &\text{ if } i \neq j\\
-\sum \limits_{l=1, l\neq i}^{m} k_{il} & \text{ if } i=j\\
\end{cases} 
\qquad \quad i,j \in \{1, \ldots, m\}
\end{equation}
According to this notation, the reaction $C_i \rightarrow C_j$ takes place in the reaction network if and only 
if $[A_k]_{ji}$ is positive, and $[A_k]_{ji} =0$ implies that $(C_i,C_j) \notin \mathcal{R}$.
Since a chemical reaction network is uniquely characterized by the  matrices $Y$ and $A_k$, we refer to a CRN by the 
corresponding pair $(Y,A_k)$.
  
If mass action kinetics is assumed, the equations governing the dynamics of the \textbf{concentrations of the 
species} in the CRN defined by the function $x:\mathbb{R} \rightarrow \mathbb{R}^n_{+}$ can 
be written in the form:
\begin{equation}\label{eq:crn_dyn}
\dot{x}=Y \cdot A_k \cdot \psi^Y(x)
\end{equation}
where $\psi^Y: \mathbb{R}_+^n \rightarrow \mathbb{R}_+^m$  is the monomial function of 
the CRN with coordinate functions
\begin{equation}\label{eq:monomials}
\psi^Y_j(x)=\prod \limits_{i=1}^{n} x_i^{[Y]_{ij}}, \quad  j\in \{1,\dots,m\}
\end{equation}

The nonnegative polynomial system \eqref{eq:kinpoly} is called a \textbf{kinetic system} if there exists a 
reaction network $(Y,A_k)$ so that its dynamics satisfies the equation \cite{Erdi1989}:
\begin{equation} \label{eq:dyneq}
M\cdot \varphi(x) = Y\cdot A_k\cdot \psi^Y(x)
\end{equation}

As it has been mentioned in the Introduction, reaction networks with different sets of complexes and 
reactions may be governed by the same dynamics. If Equation \eqref{eq:dyneq} is fulfilled, then the CRN 
$(Y,A_k)$ is called a \textbf{dynamically equivalent realization} of the kinetic system \eqref{eq:kinpoly}.

The description of the 
polynomial system \eqref{eq:kinpoly} can be transformed so that the monomial function $\varphi$ is equal to 
$\psi^Y$ (and $p=m$ holds) while the described dynamics remains the same. After the transformation and 
simplification based on the properties of polynomials, Equation \eqref{eq:dyneq} can be simplified to:
\begin{equation} \label{eq_lin}
M = Y \cdot A_k
\end{equation}

Reaction networks have another representation, which is more suitable for illustrating the structural 
properties. It is a weighted directed graph $G(V,E)$ called the \textbf{Feinberg-Horn-Jackson graph} or 
\textbf{reaction graph} for brevity \cite{Erdi1989}. The complexes are represented by the vertices, and the 
reactions by the edges. 
Let the vertices $v_i$ and $v_j$ correspond to the complexes $C_i$ and $C_j$, respectively. Then there is a 
directed edge $v_i v_j \in V(G)$ with weight $k_{ij}$ if and only if the reaction $C_i \rightarrow C_j$ takes 
place in the CRN.

\subsection{Uncertain kinetic systems} \label{ssec:uncertain}

For the uncertainty modelling, we assume that the monomial coefficients in matrix $M$ are constant but uncertain, and they belong to an $n\cdot m$ dimensional polyhedron.

\begin{rem}
In previous sections the set of uncertain parameters is noted as a polytope or a polytopic set, but from now on we use the notion of a polyhedron as well. The former one is defined as the convex hull of its vertices, while the latter one is the intersection of halfspaces, and the two definitions are not equivalent in general.
However, in the examined problems it is assumed that the parameters of the kinetic models are bounded, and a bounded polyhedron is equivalent to a bounded polytope. 
\end{rem}

We represent the matrix $M$ as a point denoted by $\widetilde{M}$ in the 
Euclidean space $\mathbb{R}^{nm}$. In the uncertain model it is assumed that the possible points 
$\widetilde{M}$ are all the points of a closed convex polyhedron $\mathcal{P}$, which is defined as the intersection of $q$ halfspaces. The boundaries of the halfspaces are hyperplanes with normal vectors 
$n_1, \ldots ,n_q \in \mathbb{R}^{nm}$ and constants $b_1, \ldots, b_q \in \mathbb{R}$. 
Applying these notations, the polyhedron $\mathcal{P}$ can be described by a linear 
inequality system as
\begin{align}\label{eq:Mpolyhedron}
\mathcal{P}= \{\widetilde{M} \in \mathbb{R}^{nm} \ | \  {\widetilde{M}}^{\top} \cdot n_i \leq b_i, \ 1 \leq i \leq q\}\}
\end{align}
For the characterization of the polyhedron $\mathcal{P}$ not only the possible values of the parameters 
should be considered, but also the kinetic property of the polynomial system. This can be ensured (see \cite{Hars1981}) by prescribing the sign pattern of the matrix $M$ as follows:
\begin{equation}
[Y]_{ij}=0  \   \Longrightarrow \   [M]_{ij}\geq 0, \quad \qquad i\in \{1, \ldots, n\}, \  j \in \{1, \ldots, m\}
\end{equation}
These constraints are of the same form as the inequalities in Equation \eqref{eq:Mpolyhedron}, for example the 
constraint $\widetilde{M}_j \geq 0$ can be written by choosing the normal vector $n_i$ to be the unit vector $-e_j^{nm}$ and $b_i$ to be the null vector $\mathbf{0}^{nm}$.

We note that there is a special case when the possible values of the parameters of the polynomial system are 
given as intervals, and the polyhedron $\mathcal{P}$ is a cuboid.

It is possible to define a set $L$ of finitely many additional linear constraints on the variables to 
characterize a special property of the realizations, for example a set of reactions to be excluded, or mass 
conservation on a given level, see e.g. \cite{Acs2016}. These constraints can affect not only the entries of 
the coefficient matrix $M$ but the Kirchhoff matrix of the realizations as well.
If the Kirchhoff matrix $A_k$ of the realization is represented by the point 
$\widetilde{A_k} \in \mathbb{R}^{m^2-m}$ storing the off-diagonal elements, and $r$ is the number of constraints in the set $L$, then the 
equations can be written in the form
\begin{equation} \label{eq:L}
\widetilde{M}^{\top} \cdot \alpha_i + {\widetilde{A_k}}^{\top} \cdot \beta_i \leq d_i
\end{equation} 

\noindent where $\alpha_i \in \mathbb{R}^{nm}$, $\beta_i \in \mathbb{R}^{m^2-m}$ and $d_i \in \mathbb{R}$ hold 
for all $i \in \{1, \ldots, r\}$. These constraints do
not change the general properties of the model, and as it will be shown in Section \ref{ssec_opt}, it can 
be modelled as a linear programming problem. 

In the case of the uncertain model, we will examine realizations assuming a fixed set of complexes.
Therefore, the known parameters are the polyhedron $\mathcal{P}$, the set $L$ of constraints and the matrix $Y$.  
Hence a \textbf{constrained uncertain kinetic system} is referred to as the triple $[\mathcal{P},L,Y]$, but we will call 
it an \textbf{uncertain kinetic system} for brevity. 

\begin{definition} A reaction network $(Y,A_k)$  is called a \textbf{realization of the uncertain kinetic 
system} $[\mathcal{P},L,Y]$ if there exists a coefficient matrix $M \in \mathbb{R}^{n\times m}$  
so that the equation $M = Y \cdot A_k$ holds, the point $\widetilde{M}$ is in the polyhedron $\mathcal{P}$ and 
the entries of the matrices $M$ and $A_k$ fulfil the set $L$ of constraints. 
Since the matrix $Y$ is fixed but the coefficients of the polynomial system can vary, this realization is 
referred to as the matrix pair $(M,A_k)$. 
\end{definition}

\subsection{Computational model} \label{ssec_opt}
Assuming a fixed set of complexes, a realization $(M,A_k)$ of an uncertain kinetic system $[\mathcal{P},L,Y]$ 
can be computed using a linear optimization framework. 

In the constraint satisfaction or optimization model, the variables are the entries of the matrix $M$ and the off-diagonal entries of the matrix $A_k$. The constraints regarding the realizations of the uncertain model can be written as follows:
\begin{align}
\  &  \widetilde{M}^{\top} \cdot n_i \leq b_i, & \  & i \in \{1, \ldots, p\} \label{eq_parameter} \\
\  &  M = Y \cdot A_k, & \  &  \   \label{eq_dyneq}\\
\  &  [A_k]_{ij} \geq 0, & \  & i \neq j, \  i,j \in \{1, \ldots, m\} \label{eq_nonneg} \\
\  &  \sum \limits_{j=1}^m [A_k]_{ij} =0, & \  &   j \in \{1, \ldots, m\} \label{eq_columns}
\end{align}
Equations \eqref{eq_parameter} ensure that the parameters of the dynamics correspond to a point of the 
polyhedron $\mathcal{P}$. Dynamical equivalence is defined by Equation \eqref{eq_dyneq}, while Equations 
\eqref{eq_nonneg} and \eqref{eq_columns} are required for the Kirchhoff property of matrix $A_k$ to be 
fulfilled. Moreover, the constraints in the set $L$ can be written in the form of Equation  \eqref{eq:L}.

The objective function of the optimization model can be defined according to the desired properties of the 
realization, for example in order to examine if the reaction $C_i \rightarrow C_j$ can be present in the 
reaction network or not, the objective can be defined as $\max [A_k]_{ji}$.

We apply the representation of realizations of the uncertain model as points 
of the Euclidean space $\mathbb{R}^{m^2-m+nm}$. The coordinates with indices $i \in \{1, \ldots, m^2-m\}$ 
characterize the Kirchhoff matrix of the realization and the remaining coordinates 
$j \in \{m^2-m+1, \ldots m^2-m+nm\}$ define the coefficient matrix $M$ of the polynomial system.  
Due to the linearity of the constraints in the computational model, the set of possible realizations of an 
uncertain kinetic system $[\mathcal{P},L,Y]$ is a convex bounded polyhedron denoted by $\mathcal{Q}$.
\section{Structural analysis of realizations of the uncertain\\ kinetic model} 

In this section we summarize some of the special structural properties of the realizations of an uncertain kinetic system $[\mathcal{P},L,Y]$.

\subsection{Superstructure property of the dense realizations} \label{sec_superstructure}

A dynamically equivalent or linearly conjugate realization of a kinetic system with a fixed set of complexes 
having maximal or minimal number of reactions is called dense or sparse realization, respectively 
\cite{Szederkenyi2009b,Johnston2012}. It is known that for any kinetic system there might be several different 
sparse realizations, however, the dense realization is structurally unique and it defines a superstructure 
among all realizations, see \cite{Johnston2012a}.

The directed graph $G(V,E)$ is called a \textbf{superstructure} with respect to a set $\mathcal{G}$ of directed 
graphs with labelled vertices, if it contains every graph in the set $\mathcal{G}$ as subgraph, and it is 
minimal under inclusion. By the definition it follows that for any set $\mathcal{G}$ there exists a 
superstructure graph and it is unique.

In the case of dynamical equivalent and linearly conjugate realizations of kinetic systems the superstructure 
is the reaction graph of a dense realization, that contains all the reaction graphs representing realizations 
of the kinetic system as subgraphs, 
not considering the edge weights. This means that the set of reactions that take place in any of the 
realizations is the same as the set of reactions in the dense realization.

Dense and sparse realizations can be introduced in the case of the uncertain model as well, that 
are useful during the structural analysis.
\begin{definition}
A realization $(M,A_k)$ of the uncertain kinetic system $[\mathcal{P},L,Y]$ is called a  \textbf{dense (sparse)} realization if it has maximal (minimal) number of reactions.
\end{definition}
It can be proved that the superstructure property holds for uncertain kinetic systems as well, and the 
proof is based on the same idea as in the non-uncertain case, see \cite{Acs2015}.  

\begin{prop} \label{prop_superstr}
A dense realization $(M,A_k)$ of an uncertain kinetic system $[\mathcal{P},L,Y]$  determines a superstructure among all realizations of the model.
\end{prop}

\begin{proof}
If the point $D$ in the polyhedron $\mathcal{Q}$ of possible realizations represents a dense realization, then 
the superstructure property is equivalent to  the property that any coordinate with index 
$i \in \{1, \ldots, m^2-m\}$ of an arbitrary point in $\mathcal{Q}$ can be positive only if the same 
coordinate of $D$ is positive.
Let us assume by contradiction that there is another realization $R \in \mathcal{Q}$ so that there is an 
index $j \in \{1, \ldots, m^2-m\}$ for which $D_j =0$ and $R_j >0$ hold.

Since the polyhedron $\mathcal{Q}$ is closed under convex combination, the point 
\[ T = c \cdot D + (1-c) \cdot R \qquad c \in (0,1) \]
is also in $\mathcal{Q}$. 
The coordinates with indices of the set $\{1,\ldots,m^2-m\}$ of all the points in 
$\mathcal{Q}$ are nonnegative, therefore such a coordinate of the convex combination is positive if the 
corresponding coordinate of $D$  or $R$ is positive. Consequently, $T$ has more positive coordinates with 
indices $j \in \{1,\ldots,m^2-m\}$ than the dense realization does,  which is a contradiction.  
\end{proof}

It follows from Proposition \ref{prop_superstr} that the structure of the dense realization is unique. 
If there were two different dense realizations, then the reaction graphs representing them would contain each 
other as subgraphs, which implies that these graphs are structurally identical.

The dense and sparse realizations are useful for checking the structural uniqueness of the uncertain model.

\begin{prop} \label{prop_uniqueness}
The dense and sparse realizations of an uncertain kinetic system $[\mathcal{P},L,Y]$ have the same number of 
reactions if and only if all realizations of the model are structurally identical.
\end{prop}

\begin{proof}
According to the definitions if in the dense and sparse realizations there is the same 
number of reactions, then in all realizations there must be the same number of reactions. Since the structure 
of the dense realization is unique, there cannot be two realizations with the maximal 
number of reactions but different structures, therefore all realizations must be structurally identical to 
the dense realization. 

The converse statement is trivial: If all the realizations of the model are structurally identical, then the 
dense and sparse realizations must have identical structures, too.
\end{proof}

\subsection{Polynomial-time algorithm to determine dense realizations} \label{sec_densealgorithm}

A dense realization of the uncertain kinetic system can be computed by the application of a 
recursive polynomial-time algorithm. The basic principle of the method is similar to the one presented in 
\cite{Acs2015}: To each reaction a realization is assigned where the reaction takes place, if it is possible. 
In general, the same realization can be assigned to several reactions. Therefore, there is no need to perform 
a separate computation step for each reaction.
The convex combination of the assigned realizations is also a realization of the uncertain model. 
If all the coefficients of the convex combination are positive then all reactions that take place in 
any of the assigned realizations are present in the convex combination as well.
Consequently, the obtained realization represents a dense realization, where all reactions are 
present that are possible.

The computation can be performed in polynomial time since it requires at most $m^2-m$ steps of LP optimization 
and some minor computation. 

\begin{rem}
It follows from the operation of the algorithm that if there are at least two realizations 
assigned to reactions as defined, then there are infinitely many dense realizations, since at least one 
coefficient of the convex combination can be chosen arbitrarily from the interval $(0,1)$. 
\end{rem}

In the algorithm the assigned realizations are represented as points in $\mathbb{R}^{m^2-m+nm}$ and 
are determined using the following procedure:

\medskip
\noindent \textbf{FindPositive$([\mathcal{P},L,Y],H)$} returns a pair $(R,B)$. The point $R \in \mathcal{Q}$  
represents a realization of the uncertain model $[\mathcal{P},L,Y]$ for which the value of the 
objective function $\sum_{j \in H} R_j$ considering a set $H \subseteq \{1, \ldots, m^2-m\}$ of indices is 
maximal. The other returned object is a set $B$ of indices where $k \in B$ if and only if $Q_k >0$. 
If there is no realization fulfilling the constraints then the pair $(\mathbf{0},\emptyset)$ is returned. \\

In the algorithm we apply the arithmetic mean as convex combination, i.e. if the number of 
the assigned realizations is $k$ then all the coefficients of the convex combination are $\frac{1}{k}$.

\begin{algorithm}[H]
\caption{(Computes a dense realization) \\Input: $[\mathcal{P},L,Y]$\\Output: $Result$ }
\begin{algorithmic}[1]
\State $H:=\{1, \ldots, m^2-m \}$ 
\State $B:= H$
\State $Result:= \mathbf{0} \in \mathbb{R}^{m^2-m+nm}$
\State $loops:= 0$
\While {$ B \neq \emptyset$}
\State $(R,B):=$ \text{FindPositive}$([\mathcal{P},L,Y],H)$
\State $Result := Result + R$
\State $H:=H \setminus B$
\State $loops:=loops+1$
\EndWhile
\State $Result := Result / loops$
\If {$Result=\mathbf{0}$}
\State {\text{There is no realization  with the given properties.}}
\Else 
\State {$Result$ \text{ is a dense realization.}} 
\EndIf
\end{algorithmic}
\end{algorithm}

\begin{prop}
The realization returned by \emph{\textbf{Algorithm 1}} is a dense realization of the uncertain kinetic 
system. 
\end{prop}

\begin{proof}
Since the set of all possible solutions can be represented as a convex polyhedron, 
the point $Result$ computed as the convex combination of realizations is indeed a realization of the 
uncertain kinetic system $[\mathcal{P},L,Y]$. 
Let us assume by contradiction that the returned point $Result$ does not represent the dense realization. Then 
there is a reaction $(C_i,C_j)$ which is present in the dense realization but it does not take place in 
$Result$. By the operation of the algorithm it follows that there must be a realization assigned to the 
reaction $(C_i,C_j)$, consequently this reaction takes place in the realization computed as the convex 
combination of the assigned realizations as well. This is a contradiction.
\end{proof}

\subsection{Core reactions of uncertain models} \label{sec_core}

A reaction is called \textbf{core reaction} of a kinetic system if it is present in every realization of the 
kinetic system \cite{Szederkenyi2011c}. It is possible that there are no core reactions, but there can be several of them as well.
If all the realizations are structurally identical, then by Proposition \ref{prop_uniqueness} it follows that 
each reaction is a core reaction.
The notion of core reactions can be extended to the case of uncertain models in a straightforward way.

\begin{definition}
A reaction $C_i \rightarrow C_j$ is called a \textbf{core reaction} of the uncertain kinetic system 
$[\mathcal{P},L,Y]$  if it is present in each realization of the model, considering all possible coefficient 
matrices $M$ for which $\widetilde{M} \in \mathcal{P}$ holds.
\end{definition}

Let $[\mathcal{P},L,Y]$ and $[\mathcal{P'},L,Y]$ be two uncertain kinetic systems considering the same sets of 
complexes and additional linear constraints so that the polyhedron $\mathcal{P'}$ is a subset of 
$\mathcal{P}$. If the sets of core reactions in 
the models are denoted as $C_{\mathcal{P}}$ and $C_{\mathcal{P'}}$, respectively, then 
$C_{\mathcal{P}} \subseteq C_{\mathcal{P'}}$ must hold. 
This property holds even if $\mathcal{P'}$ is a single point in $\mathbb{R}^{nm}$ and $[\mathcal{P'},L,Y]$ is 
a kinetic system defined as an uncertain kinetic system.

The set of core reactions of an uncertain kinetic system can be computed using a polynomial-time algorithm.  
This method has been first published in \cite{Tuza2015} for a special case, where the coefficients of the 
polynomial system have to be in predefined intervals, therefore the polyhedron $\mathcal{P}$ is a cuboid. 
Since the model applies only the property that all the constraints characterizing the model are linear, it can 
be applied without any modification to uncertain kinetic systems as well.

The question whether a certain reaction is a core reaction of a kinetic model or not, can be answered by 
solving a linear optimization problem. If this question has to be  decided for all possible reactions, the 
computation can be done more effectively than doing separate optimization steps for every reaction. The idea 
is to minimize the sum of variables representing the off-diagonal entries of the Kirchhoff matrix. Generally, several variables in the minimized sum are zero in the computed realization, which means that the 
reactions corresponding to these variables are not core reactions. This step is repeated with the remaining 
set of variables until the computation does not return any non-core reactions. Finally, the remaining variables need to be 
checked one-by-one.

In the algorithm we refer to sets of indices corresponding to the off-diagonal entries of the 
Kirchhoff matrix $A_k$ by their characteristic vectors. The set $B \subseteq \{1, \ldots, m^2-m\}$ represented by the vector $b \in \{0,1\}^{m^2-m}$, which is defined as

\begin{equation}
b_i= \begin{cases} 
1 &\text{ if } i \in B\\
0 & \text{ if } i \notin B\\
\end{cases} 
\end{equation}

The procedure applied during the computation is more formally the following: 

\bigskip

\noindent \textbf{FindNonCore}$([\mathcal{P},L,Y],b)$ computes a realization of the uncertain kinetic system 
$[\mathcal{P},L,Y]$ represented as a point $R \in \mathbb{R}^{m^2-m+nm}$, for which the sum of the coordinates 
with indices in the set $B \in \{1, \ldots, m^2-m\}$ is minimal. 
The procedure returns the vector $c$, the characteristic vector of the set $C$ which 
contains the indices corresponding to zero entries of the Kirchhoff matrix of the realization $R$, i.e. 
$C \subseteq \{1, \ldots, m^2-m\}$ and $\left[ i \in C \   \Longleftrightarrow \   R_i=0 \right]$. 

\bigskip

\noindent We also need to utilize some operations on the sets represented by their characteristic vectors:
\begin{tabular}{ll}
$b*c$ & represents the set $B \cap C$, i.e. it is an element-wise `logical and'\\
$\overline{c}$ & represents the complement of the set $C$, 
i.e. it is an element-wise negation. 
\end{tabular}

\begin{algorithm}[H]
\caption{(Computes the set of core reactions) \\Inputs: $[\mathcal{P},L,Y]$\\Output: $b$  }
\begin{algorithmic}[1]
\State $b:=\mathbf{1}^{m^2-m}$ 
\State $c:=b$
\While{$c \neq \mathbf{0}$}
\State $c:=$ \text{FindNonCore}$([\mathcal{P},L,Y],b)$
\State $c:= c * b$  
\State $b:= b * \overline{c}$
\EndWhile
\For {$i= 1$ to $m^2-m$} 
\If {$b_i \neq 0$}
\State $c:=$ \text{FindNonCore}$([\mathcal{P},L,Y],e_i^{m^2-m})$
\State $b:= b * \overline{c}$
\EndIf
\EndFor
\If {$b = \mathbf{0}$}
\State {There are no core reactions of the model $[\mathcal{P},L,Y]$.}
\Else 
\State{The vector $b$ characterizes the core reactions of the model $[\mathcal{P},L,Y]$.}
\EndIf
\end{algorithmic}
\end{algorithm}

\begin{prop} \emph{\textbf{Algorithm 2}} computes the set of core reactions of an uncertain kinetic system $[\mathcal{P},L,Y]$ in polynomial time.
\end{prop}

\begin{proof}
Let us assume by contradiction that the algorithm does not return the proper set of core reactions. There can be two different types of error: 

$a)$ Let us assume that there is an index $i$ for which the corresponding reaction is a core reaction, but 
according to the algorithm it is not. In this case there must be a realization $R$ computed by the algorithm 
so that $R_i$ is zero. This is a contradiction.

$b)$ Let us assume that there is an index $j$ for which the corresponding reaction is not a core reaction but 
the algorithm returns the opposite answer. Consequently, after the while loop of the computation (from line 8) the 
coordinate $b_j$ must be equal to $1$. Then the remaining possible core reactions are examined one by one, 
therefore the procedure FindNonCore$([\mathcal{P},L,Y],e_j^{m^2-m})$ is also applied. According to the assumption the 
realization $R$ computed by the procedure must be so that $R_j$ is zero, which also yields a contradiction.

The computation according to the algorithm can be performed in polynomial time, since it requires the solution 
of at most $m^2-m$ LP optimization problems and some additional minor computation steps.
\end{proof}

\section{Algorithm to determine all possible reaction graph structures of uncertain models} \label{sec_alg}
In this section we introduce an algorithm for computing all possible reaction graph structures of an 
uncertain kinetic system $[\mathcal{P},L,Y]$. The proposed method is an 
improved version  of the algorithm published in \cite{Acs2017}, where all the optimization steps can be done 
parallelly. We also give a proof of the correctness of the presented method. Before presenting the pseudocode of the algorithm, we give a brief explanation of its data structures and operating principles.

We represent reaction graph structures by binary sequences, where each entry encodes the presence or lack of a 
reaction.
During the algorithm, all data (i.e. the Kirchhoff and the coefficient matrices) of the realizations are 
computed, but only the binary sequences encoding the directed graph structures are stored and returned as 
results.

According to the superstructure property described in Proposition 
\ref{prop_superstr},  only the reactions belonging to the dense realization need representation and storage. 
Moreover, if there are core reactions as well, then the coordinates corresponding to these can also be omitted. 
Both sets can be computed in polynomial time as it has been presented in Sections \ref{sec_densealgorithm} and 
\ref{sec_core}.

Let us refer to the set of reactions in the dense realization and the set of core reactions in the uncertain 
kinetic system $[\mathcal{P},L,Y]$ as $D_{\mathcal{P}}$ and $C_{\mathcal{P}}$, respectively. 
Then a realization of the uncertain model $[\mathcal{P},L,Y]$ can be represented by a binary 
sequence $R$ of length $z$, where $z$ is the size of the set $D_{\mathcal{P}} \setminus C_{\mathcal{P}}$ of 
non-core reactions in the dense realization.
To define the binary sequence $R$ it is necessary to fix an ordering on the set of non-core reactions.  
The coordinate $R_i$ is equal to $1$ if and only if the $i$th non-core reaction is 
present in the realization, otherwise it is zero.

It is easy to see that knowing its structure, a realization can be determined in polynomial time: 
For each reaction $C_i \rightarrow C_j$ which is known not to be present in the realization the constraint 
$[A_k]_{ji}=0$ needs to be added to the constraint set $L$, and a dense realization of the (constrained) model 
has to be computed. 
Since it is known that there exists a realization where all non-excluded reactions take place, all of them 
have to be present in the computed constrained dense realization, consequently it will have exactly the 
prescribed structure. 

During the computation the initial substrings of the binary sequences have a special role. Therefore, for all 
$k \in \{1, \ldots z\}$ a special equivalence relation $=_k$ is defined on the binary sequences. We say that
$R =_k W$ holds if for all  $i \in \{1, \ldots k\}$ the coordinate $R_i$ is equal to $W_i$. 
The equivalence class of the relation $=_k$ that contains the sequence $R$ as a representative is referred to 
as $C_k(R)$. (We note that in general there are several representatives of an equivalence class.)
The elements of an equivalence class $C_k(R)$ can be characterized by a set of linear constraints added to the 
model. According to this property and Proposition \ref{prop_superstr}, the dense realization in $C_k(R)$ 
determines a superstructure among all the realizations in the same set.
The procedure FindRealization applied during the algorithm computes dense realizations of the uncertain 
model determined by the initial substrings. A realization is referred to as a pair $(R,k)$ if the corresponding realization represents the dense realization in $C_k(R)$. The realizations represented by such pairs get stored for some time in a stack $S$, the command `push $(R,k)$ into $S$' puts the pair $(R,k)$ into the stack and `pop from $S$' takes a pair out of the stack and returns it. 
The number of elements in the stack $S$ is denoted by $size(S)$.

The result of the entire computation is collected in a binary array called $Exist$, where all the computed 
graph structures are stored. The indices of the elements are the sequences as 
binary numbers, and the value of element $Exist[R]$ is equal to $1$ if and only if 
a realization with the structure encoded by $R$ has been found. 

Considering the data structures, the main difference between the proposed method and the algorithm presented 
in \cite{Acs2017} is that the sequences encoding the reaction graph structures are stored in only one stack in 
our current solution. Furthermore, the optimization steps using the sequences popped from this stack can be run in parallel. However, in this case the use of the binary array $Exist$ is necessary.

Within the algorithm we repeatedly apply two subroutines:

\medskip

\noindent \textbf{FindRealization}$((R,k),i)$ computes a dense realization of the uncertain kinetic system 
$[\mathcal{P},L,Y]$, for which the representing binary sequence $W$ is in $C_k(R)$, and for every index 
$j \in \{k+1,\ldots, i\}$ the coordinate $W_j$ is zero. 
It is possible that 
among the first $k$ coordinates there are more zeros than required, therefore the computed sequence $W$ is 
compared to the sequence $R$. The procedure returns the sequence $W$ only if $W=_k R$ holds, otherwise $-1$ is returned. If the optimization task is infeasible then the returned object is also $-1$.

\medskip

\noindent \textbf{FindNextOne}$((R,k))$ returns the smallest index $i$  for which $k<i$ and $R_i=1$ hold. If 
there is no such index, i.e. $R_j$ is zero for all $k<j$, then it returns $z+1$, where we recall that $z$ is the length of the sequences that encode the graph structures.

\medskip

Let the sequence $D=\mathbf{1}$ represent the dense realization. Then the pseudocode of the algorithm for computing all possible graph structures can be given as follows.
\begin{algorithm}[H] \label{alg:all}
\caption{(Computes all reaction graph structures of an uncertain kinetic system) \\Inputs: $[\mathcal{P},L,Y],D,z$\\Output: $Exist$ }
\begin{algorithmic}[1]
\State push $(D,0)$ into $S$ 
\State $Exist[D]:=1$
\While{$size(S)>0$}
\State $(R,k):=\text{pop from }S$
\State $i:=\text{FindNextOne}((R,k))$
\If {$i<z$}
\State push $(R,i)$ into $S$
\EndIf
\While {$i < z$}
\State $W:=\text{FindRealization}((R,k),i)$
\If {$W < 0$}
\State BREAK
\Else 
\State $i:=\text{FindNextOne}(W,i)$
\State $Exist[W]:=1$
\If {$i<z$}
\State push $(W,i)$ into $S$
\EndIf
\EndIf
\EndWhile 
\EndWhile
\end{algorithmic}
\end{algorithm}
Using the description of the algorithm, we can give formal results about its main properties.
\begin{prop} \label{prop:all_reals} \emph{\textbf{Algorithm 3}} computes all possible reaction graph structures representing realizations of an uncertain kinetic system $[\mathcal{P},L,Y]$. 
\end{prop}

\begin{proof} 
Let us assume by contradiction that there is a realization of the uncertain kinetic system $[\mathcal{P},L,Y]$ 
represented by the sequence $V$ which is not returned by \textbf{Algorithm 3}. 
Let $R$ be another sequence that was stored in the stack $S$ as $(R,p)$ at some point during the computation,  
for which $V =_p R$ holds and $p$ is the greatest such number. 
If $p=0$ then $D$ is suitable to be $R$, and by the operation of the algorithm it follows that $p<z$ holds. (If 
$p$ were equal to $z$, then $V$ would be equivalent to $R$ which is a contradiction.)

There is a point during the computation when $(R,p)$ is popped out from the stack $S$.   
Let us assume that $\text{FindNextOne}(R,p)$ returns $i$ and $\text{FindNextOne}(V,p)$ returns $j$. 
In this case $i \leq j$ must hold since $R$ represents the superstructure in $C_p(R)$ and if $i$ were equal to 
$j$ then $p$ would not be maximal.

For the examination of sequence $R$, the procedure $\text{FindRealization}((R,p),i)$ is applied first 
(line 10), and it must return a valid sequence $W_1$, since its constrains are fulfilled by the realization 
$V$ as well. 
If $\text{FindNextOne}(W_1,p)$ is $j_1$ then $j_1 \leq j$ must hold, since $W_1$ represents the
superstructure in $C_i(W_1)$ and $V$ is also in $C_i(W_1)$. If $j_1$ was equal to $j$ then $p$ 
would not be maximal. Otherwise, 
the computation can be continued by calling the procedure $\text{FindRealization}((R,p),j_1)$. It must 
return a valid sequence $W_2$ for which we get that $\text{FindNextOne}(W_2,p)=j_2 \leq j$ holds by applying 
similar reasoning as above. 

These steps must lead to contradiction either by $p$ not being maximal or by creating an infinite increasing 
sequence of integers that has an upper bound.

It follows that every possible reaction graph structure that represents a realization of the uncertain kinetic 
system $[\mathcal{P},L,Y]$ is returned by the algorithm.
\end{proof}

\begin{rem} Since the calculations of procedure \emph{FindRealization}$((R,k),i)$ are independent of the 
results of previous calls of the same procedure, the order of the calls is irrelevant regarding the result 
of the entire computation.
\end{rem}

\begin{rem} The proof of \emph{Proposition 3.2} in \emph{\cite{Acs2017}} can be applied for verifying the property 
that during the computation according to \emph{\textbf{Algorithm 3}} every reaction graph 
structure is returned only once. 
\end{rem}

\begin{rem} We can also give an upper bound to the number of required optimization steps by considering the 
realizations $(R,k)$ regarding $k$. For all $k$ the number of possible realizations $R$ stored in the stack 
$S$ is at most $2^k$. When such a realization is popped from the stack the required optimization steps is at 
most $z-k$. Consequently, a rough upper bound to the number of optimization steps 
required during \emph{\textbf{Algorithm 3}} can be given as $\sum \limits_{k=0}^{z-1} 2^k (z-k)$. 
\end{rem}

\section{Illustrative examples}
In this section we demonstrate the operation of the algorithms presented in this paper on two examples in case 
of different degrees and types of uncertainties, and even in the case of additional linear constraints.

\subsection{Example 1: a simple kinetic system} \label{ssec:example1}
The model that serves as a basis for this example was presented previously in \cite{Szederkenyi2011,Acs2016}. 
The uncertain model is generated using the kinetic system

\begin{align}
\  & \dot{x}_1 = 3c_1 \cdot x_2^3-c_2 \cdot x_1^3  \nonumber \\
\  & \dot{x}_2 = -3c_1 \cdot x_2^3+c_2 \cdot x_1^3, \label{eq:simplesys} 
\end{align}
where $c_1,c_2>0$.
We consider realizations on a fixed set $\mathcal{C}=\{C_1,C_2, C_3\}$ of 
complexes, where the complexes $C_1 = 3X_2$, $C_2 = 3X_1$, $C_3 = 2X_1+X_2$ are formed of the species $X_1$ 
and $X_2$. It follows that the characterizing matrices $Y$ and $M$ of the kinetic system referred to as 
$[M,Y]$ are
\begin{center}
$Y= \begin{bmatrix}
0 & 3 & 2 \\
3 & 0 & 1
\end{bmatrix} 
\qquad 
M= \begin{bmatrix*}[r]
3c_1 & -c_2 & 0 \\
-3c_1 & c_2 & 0
\end{bmatrix*} $
\end{center}
During the numerical computations the parameter values $c_1 = 1$ and $c_2 = 2$ were applied. 

\bigskip

\noindent \textit{A. Uncertainty defined by independent intervals}\\
This model represents a special case in the class of uncertain kinetic systems defined in Section 
\ref{ssec:uncertain}, since the possible values of every coefficient of the kinetic system are determined by  
independent upper and lower bounds that are defined as relative distances. Let us represent the entry  $[M]_{ij}$ of the coefficient matrix $M$ by the coordinate $\widetilde{M}_l$ of the point $\widetilde{M} \in \mathbb{R}^6$. Moreover, let the relative distances of the upper and lower bounds of $\widetilde{M}_l$ be given by the real constants $\gamma_l$ and $\rho_l$ from the interval $[0,1]$, respectively. Then the 
equations defining the polyhedron $\mathcal{P}_A \subset \mathbb{R}^6$ of the uncertain parameters can be 
written in terms of the coordinates $\widetilde{M}_l$ as
\begin{align}
\widetilde{M_{\  }}^{\top} \cdot e_l^6 \leq (1+ \gamma_l) \cdot [M]_{ij}\label{eq:uncert_1}\\
\widetilde{M_{\  }}^{\top} \cdot (-e_l^6) \leq (\rho_l - 1) \cdot [M]_{ij}\label{eq:uncert_2}
\end{align}
In the examined uncertain kinetic system $[\mathcal{P}_A,L,Y]$ no additional linear 
constraints are considered, i.e. $L = \emptyset$.

In \cite{Acs2016} all possible reaction graphs -- with the indication of the reaction rate constants defined 
as functions of the parameters $c_1$ and $c_2$ -- representing dynamically equivalent realizations of the 
kinetic system $[M,Y]$ have been presented. 
Obviously, these structures must appear among the realizations of the uncertain kinetic model $[\mathcal{P}_A,\emptyset,Y]$ as well, but  there might be additional possible structures among the realizations of the uncertain kinetic system.

Interestingly, the result of the computation was that in the case of any degree of uncertainty ($\gamma_l, \rho_l \in [0,1)$ for all $ l\in \{1, \ldots 6\}$), the sets of possible 
reaction graph structures of the uncertain model $[\mathcal{P}_A,\emptyset,Y]$ and that of the non-uncertain system 
$[M,Y]$ are identical. This result might be contrary to expectations, but for this small example it is easy to prove that the obtained graph structures are indeed correct for   
all positive values of the parameters $c_1$ and $c_2$. For this, we divide the computation into smaller 
steps. 

It has been shown in \cite{Acs2017} that in the case of dynamically equivalent realizations the computation  
can be done column-wise (since the $j$th 
column of matrix $A_k$  depends only on the $j$th column of matrix $M$). These computations can be performed 
separately, and all the possible reaction graph structures can be constructed by choosing a column structure 
for every index $j \in \{1, \ldots,m\}$ and building the Kirchhoff matrix $A_k$ of the realization from them.
Consequently, if in the case of the $j$th column the number of different structures is $p_j$, then the number 
of structurally different realizations is $\prod \limits_{j=1}^m p_j$. 

First the original kinetic system $[M,Y]$ is examined. 
To make the notations less complicated, the entries of the Kirchhoff matrix are denoted by the corresponding reaction rate coefficients, i.e. $[A_k]_{ij} = k_{ji}$ for all $i,j \in \{1,2,3\}$, $i\neq j$. 

In the case of the first column we get:
\begin{equation}
Y \cdot 
\begin{bmatrix}
-k_{12}-k_{13} \\ k_{12} \\ k_{13}
\end{bmatrix} 
= \begin{bmatrix}
3c_1\\
-3c_1
\end{bmatrix} \quad k_{12}, k_{13} \in \mathbb{R}^+ \qquad \Longrightarrow \qquad  k_{12} \in [0,c_1], \ k_{13}=\frac{3}{2}c_1 - \frac{3}{2}k_{12}
\end{equation}

\noindent It can be seen that for every positive value of the parameter $c_1$ the two corresponding reaction 
rates can realize 3 of the $2^2 =4$ possible structurally different solutions. Both can be positive, or either one 
can be positive while the other one is zero. (Possible outcomes are for example: 
$k_{12}= \frac{1}{2}c_1 , k_{13}=\frac{3}{4}c_1$ or $k_{12}= 0, k_{13}=\frac{3}{2}c_1$ or 
$k_{12}= c_1 , k_{13}=0$.) The fourth case, when both $k_{12}$ and $k_{13}$ are zero is 
possible only when $[M]_{.1}= [0 ~~ 0]^{\top}$, which requires the corresponding parameters of uncertainty 
$\rho_i$ to be at least one. %

In the case of the second column, 3 of the 4 possible outcomes can be realized and a similar reasoning 
can be applied:
\begin{equation}
Y \cdot 
\begin{bmatrix}
k_{21} \\ -k_{21}-k_{23} \\ k_{23}
\end{bmatrix} 
= \begin{bmatrix}
-c_2\\
c_2
\end{bmatrix} \quad k_{21},k_{23} \in \mathbb{R}^+ \qquad \Longrightarrow \qquad  k_{21} \in (0,\frac{c_2}{3}), \quad k_{23}=c_2-3k_{21} 
\end{equation}
In the third column there is no uncertainty because  there are only zero entries in $[M]_{.3}$. 
Consequently,  in the case of $[A_k]_{.3}$ only 2 solutions are possible. The two corresponding reactions can 
either be both present or both missing.
\begin{equation}
Y \cdot 
\begin{bmatrix}
k_{31} \\ k_{32} \\ -k_{31}-k_{32}
\end{bmatrix} 
= \begin{bmatrix}
0\\
0
\end{bmatrix} \quad k_{31},k_{32} \in \mathbb{R}^+ \qquad \Longrightarrow \qquad  k_{31} \in \mathbb{R}^+, \quad k_{32}=2k_{31} 
\end{equation}
It follows from the above computations that the number of possible reaction graph structures is $3\cdot 3 \cdot 2=18$, and the generated 
structures are identical to the ones presented in \cite{Acs2016}.
This number could be larger only if all the reaction rates in the first or second column of $A_k$ can be zero, 
but this requires the entries in the corresponding column $[M]_{.1}$ or $[M]_{.2}$ to be zero.

\bigskip

\noindent \textit{B. Uncertainty defined as a general polyhedron}\\
Now we examine the uncertain kinetic system that was also generated from the kinetic system $[M,Y]$, but the 
set $\mathcal{P}_B$ of possible coefficients is defined as a more general polyhedron. 

If the matrix $M$ of coefficients is represented by the point $\widetilde{M} \in \mathbb{R}^6$ so that 
$\widetilde{M_{\  }}^{\top} = [M_{11},M_{12},M_{13},M_{21},M_{22},M_{23}]$, then let the equations determining the polyhedron $\mathcal{P}_B$ be the following:
\begin{align}
\widetilde{M_{\  }}^{\top} \cdot (-e_1^6) & \leq 0\nonumber\\
\widetilde{M_{\  }}^{\top} \cdot (-e_5^6) & \leq 0\nonumber\\
\widetilde{M_{\  }}^{\top} \cdot e_3^6 & = 0\nonumber\\
\widetilde{M_{\  }}^{\top} \cdot e_6^6 & = 0\label{eq:poly_unc}\\
\widetilde{M_{\  }}^{\top} \cdot [1,1,0,1,1,0]^{\top} & =0\nonumber\\
\widetilde{M_{\  }}^{\top} \cdot [0,-1,0,-1,0,0]^{\top} & \leq 7\nonumber\\
\widetilde{M_{\  }}^{\top} \cdot [-1,0,0,0,1,0]^{\top} & \leq -1\nonumber
\end{align}
In this case, again, no additional linear constraints are considered in the uncertain model, i.e. we 
examine the uncertain model $[\mathcal{P}_B, \emptyset, Y]$. The computation of all possible reaction graph 
structures shows that in addition to the structures realizing the non-uncertain kinetic system $[M,Y]$, there are 6 more possible 
structures, presented in Figure \ref{fig:triangle}.

It can be seen that the point $\widetilde{M}_1^{\top} = [3,-2,0,-3,2,0]$ corresponding to the original kinetic 
system is in the polyhedron $\mathcal{P}_B$, therefore the 18 structures determined by its realizations must 
be among the realizations of the uncertain kinetic system. 
Then we can apply a reasoning similar to that in Section 5.1.$A$. 
Since the entries in column $[M]_{.3}$ are all zero in every point of $\mathcal{P}_B$, only the two outcomes 
that appear in the case of the original kinetic system $[M,Y]$ are possible in the case of this column.
The uncertain model can have more realizations than the original kinetic system only if all the reaction rates 
in at least one of the columns $[A_k]_{.1}$ or $[A_k]_{.2}$ can be zero. This is possible only if all the 
entries in $[M]_{.1}$ or $[M]_{.2}$ are zero.
From the constraints of the polyhedron $\mathcal{P}_B$ it follows that $[M]_{11} \geq 1$, consequently
the column $[M]_{.1}$ cannot be zero. But $[M]_{.2}$ can have only zero entries, for example the point 
$\widetilde{M}_2 = [3,0,0,-3,0,0] ^{\top} \in \mathcal{P}_B$ satisfies this property. 
For the columns of the matrices $M$ and $M_2$ the following hold: $[M_2]_{.1} = [M]_{.1}$ and 
$[M_2]_{.3} = [M]_{.3}$. Therefore, for the first and third columns of $A_k$ there are 3 and 2 possible 
outcomes, respectively. Since in the case of the second column there is one additional possible outcome, the 
number of further reaction graph structures (compared to the original kinetic system $[M,Y]$) is 
$3\cdot 2= 6$. It is easy to see that these are exactly the ones presented in Figure $\ref{fig:triangle}$ with 
the indicated reaction rate coefficients for an arbitrary  $p > 0$.

\begin{figure}[H]
\begin{center}
    \includegraphics[width=0.88\textwidth]{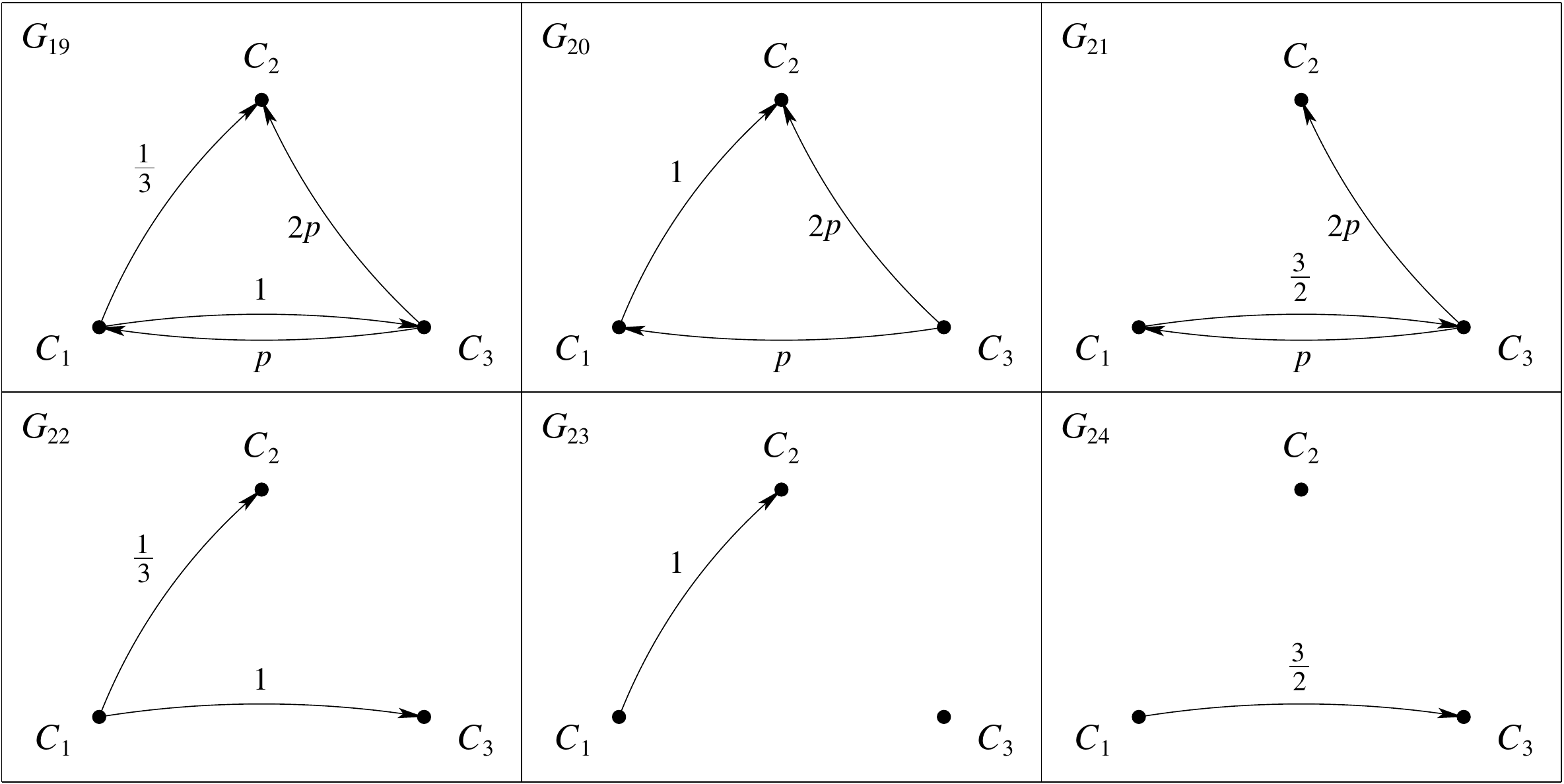}
\end{center}
	\caption{Possible reaction graph structures of the uncertain kinetic system 
	$[\mathcal{P}_B, \emptyset, Y]$ in addition to the realizations of the kinetic system $[M,Y]$ published 
	in \cite{Acs2016}.}
  \label{fig:triangle}
\end{figure}

\subsection{Example 2: G-protein network}
The G-protein (guanine nucleotide-binding protein) cycle has a key role in several intracellular signalling 
transduction pathways. The G-protein located on the intracellular surface of the cell membrane is activated by 
the binding of specific ligand molecules to the G-protein coupled receptor of the extracellular membranes 
surface. The activated G-protein dissociates to different subunits which take part in intracellular signalling 
pathways. After the termination of the signalling mechanisms, the subunits become inactive and bind into each 
other \cite{Lodish2000}.

We examined the structural properties of the yeast G-protein cycle using the model published in \cite{Yi2003}. 
The model involves a so-called heterotrimeric G-protein containing three different subunits. In response to 
the extracellular ligand binding, the protein dissociates to G-$\alpha$ and G-$\beta \gamma$ subunits, where 
the active and inactive forms of the G-$\alpha$ subunit can also be distinguished. 

The reaction network model involves the following species: $R$ and $L$ represent the receptor and the 
corresponding ligand, respectively, $RL$ refers to the ligand-bound receptor, $G$ is the G-protein located on 
the intracellular membrane surface, $G_a$ and $G_d$ denote the active and the inactive forms of the G-$\alpha$ 
subunit and $G_{bg}$ is the G-$\beta \gamma$ subunit.

The model can be characterized as a chemical reaction network $(Y,A_k)$, where the structures of the complexes 
and the reactions are defined by the complex composition matrix $Y \in \mathbb{R}^{7 \times 10}$ and the 
Kirchhoff matrix $A_k \in \mathbb{R}^{10 \times 10}$ as follows:

\begin{align*}
Y & = \begin{bmatrix}
1 & 0 & 1 & 0 & 0 & 0 & 0 & 0 & 0 & 0\\
0 & 0 & 1 & 0 & 0 & 0 & 0 & 0 & 0 & 0\\
0 & 1 & 0 & 0 & 0 & 0 & 1 & 0 & 0 & 0\\
0 & 0 & 0 & 1 & 0 & 0 & 1 & 0 & 0 & 0\\
0 & 0 & 0 & 0 & 1 & 0 & 0 & 1 & 0 & 0\\
0 & 0 & 0 & 0 & 0 & 0 & 0 & 1 & 1 & 0\\
0 & 0 & 0 & 0 & 0 & 1 & 0 & 0 & 1 & 0\end{bmatrix}\\
\   & \   \\
A_{k} & = \begin{bmatrix}
-0.4 & 0 & 0 & 0 & 0 & 0 & 0 & 0 & 0 & 4000\\
0 & -14 & 0.322 & 0 & 0 & 0 & 0 & 0 & 0 & 0\\
0 & 10 & -0.322 & 0 & 0 & 0 & 0 & 0 & 0 & 0\\
0 & 0 & 0 & 0 & 0 & 0 & 0 & 0 & 1000 & 0\\
0 & 0 & 0 & 0 & -11000 & 0 & 0 & 0 & 0 & 0\\
0 & 0 & 0 & 0 & 11000 & 0 & 0 & 0 & 0 & 0\\
0 & 0 & 0 & 0 & 0 & 0 & -0.01 & 0 & 0 & 0\\
0 & 0 & 0 & 0 & 0 & 0 & 0.01 & 0 & 0 & 0\\
0 & 0 & 0 & 0 & 0 & 0 & 0 & 0 & -1000 & 0\\
0.4 & 4 & 0 & 0 & 0 & 0 & 0 & 0 & 0 & -4000 \end{bmatrix}
\end{align*}
The kinetic system that is realized by the model is $\dot{x} = M \cdot \psi^Y = Y\cdot A_k \cdot \psi^Y$, i.e. 
$M = Y\cdot A_k \in \mathbb{R}^{7 \times 10}$. 
The reaction graph structure of the G-protein model can be seen in Figure \ref{fig:Gprotein} with the 
indication of the linkage classes. (The linkage classes are the undirected connected components of the 
reaction graph.)

\begin{figure}[H]
\begin{center}
    \framebox{\includegraphics[width=0.7\textwidth]{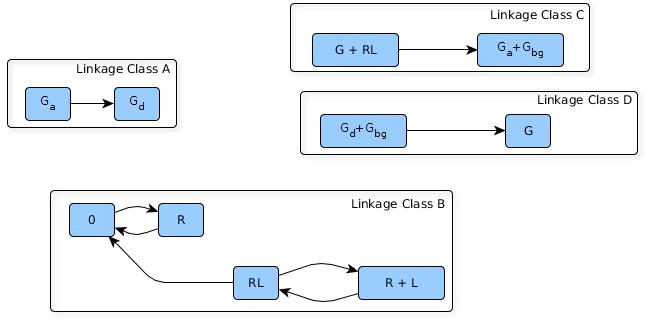}}
\end{center}
  \caption{Reaction graph structure of the heterotrimeric G-protein cycle.}
  \label{fig:Gprotein}
\end{figure}
The computation of all possible reaction graph structures and the solution of the linear equations shows that 
the heterotrimeric G-protein cycle with the given parametrization is not just structurally but also 
parametrically unique. Thus the prescribed dynamics without uncertainty cannot be realized by any other set of 
reactions or different reaction rate coefficients using the given set of complexes. 

\bigskip

\noindent \textit{A. Uncertainty defined with independent relative distance intervals}

\bigskip

We have examined the uncertain kinetic systems defined by relative parameter uncertainty as it was presented 
in Section \ref{ssec:example1}. 

First we examined the uncertain model $[\mathcal{P}_{0.1},\emptyset,Y]$, where the uncertainty coefficients 
$\gamma_l$ and $\rho_l$ for all $l \in \{1, \ldots , 70\}$ are $0.1$ and there are no additional linear 
constraints in the model. By computing all possible reaction graph structures and the set of core reactions of 
this uncertain kinetic system, we obtained that all the reactions in the original G-protein cycle are core 
reactions. Moreover, in the dense realization there are 10 further reactions, and these can be present in the 
realization independently of each other. Consequently, the total number of different graph structures is 
$2^{10}=1024$. Figure \ref{fig:distribution} shows the number of possible reaction graph structures with 
different number of reactions. The dense realization for this case is shown in Figure \ref{fig:Gprotein_M1}.

\begin{figure}[H]
\begin{center}
    \includegraphics[width=0.6\textwidth]{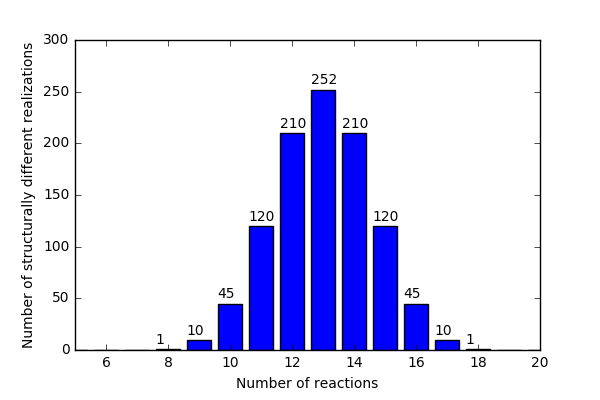}
\end{center}
  \caption{The number of structurally different realizations of the uncertain kinetic system $[\mathcal{P}_{0.1},\emptyset,Y]$ with different number of reactions.}
  \label{fig:distribution}
\end{figure}

If we increase the relative uncertainty to 0.2, we obtain the uncertain kinetic system $[\mathcal{P}_{0.2},\emptyset,Y]$ with $\gamma_l=\rho_l=0.2$ for $i=1,\dots,70$. In this case, the reaction $RL \  \rightarrow 0$ 
is no longer a core reaction, and it can also be added or removed independently of all 
other reactions (which remain independent of each other). Therefore,
the number of possible structures becomes $2^{11}=2048$.

\begin{figure}[H]
\begin{center}
    \framebox{\includegraphics[width=0.6\textwidth]{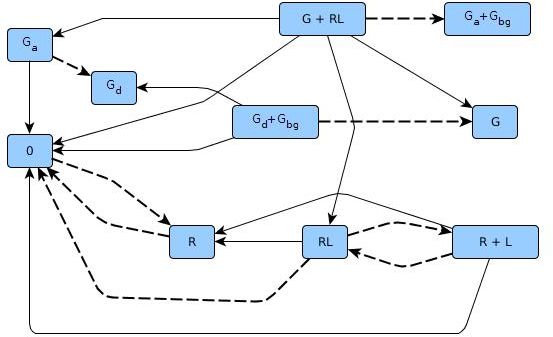}}
\end{center}
  \caption{Reaction graph structure of the dense realization of the uncertain kinetic system $[\mathcal{P}_{0.1},\emptyset,Y]$. The core reactions are drawn with dashed lines.}
  \label{fig:Gprotein_M1}
\end{figure}

\bigskip

\noindent \textit{B. Constrained uncertain model}

\bigskip

We have also examined the possible structures in the case of constrained uncertain models. The set $L_1$ of 
constraints prohibits every reaction between different linkage classes. It can be seen in Figure 
\ref{fig:Gprotein_L1} that the dense realization of the uncertain kinetic system $[\mathcal{P}_{0.1},L_1,Y]$  
has 3 reactions that are exactly the ones that are present in the dense realization of 
$[\mathcal{P}_{0.1},\emptyset,Y]$ and do not connect different linkage classes. 
These reactions are independent of each other, therefore the number of structurally different realizations is 
$2^3=8$ in the case of the uncertain kinetic system $[\mathcal{P}_{0.1},L_1,Y]$ and $2^4=16$ for the model 
$[\mathcal{P}_{0.2},L_1,Y]$.  The sets of core reactions are the same as in the case of the unconstrained model for both degrees of uncertainty.
\begin{figure}[H]
\begin{center}
    \framebox{\includegraphics[width=0.6\textwidth]{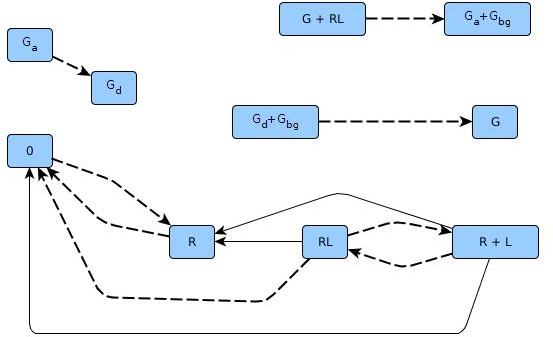}}
\end{center}
  \caption{Reaction graph structure of the dense realization of the uncertain kinetic system $[\mathcal{P}_{0.1},L_1,Y]$. The core reactions are drawn with dashed lines.}
  \label{fig:Gprotein_L1}
\end{figure}
The independence of non-core reactions is a special property of the studied uncertain model. As a consequence of this and the superstructure property of the dense realization, the dense realization of the constrained model will contain each reaction of the unconstrained model that is not excluded by the constraints. 

We emphasize that the dense realizations in the above example contain all mathematically possible reactions that can be compatible with the studied uncertain models. If, using prior knowledge, the biologically non-plausible reactions are excluded and/or certain relations between model parameters are ensured via linear constraints, then the described methodology is still suitable to check the structural uniqueness of the resulting uncertain kinetic model.

\section{Conclusion}
The set of reaction graph structures realizing uncertain kinetic models was studied in this paper. For this, an uncertain polynomial model class was introduced, where the coefficients of monomials belong to a polytopic set. Thus, an uncertain kinetic model includes a set of kinetic ordinary differential equations. Using the convexity of the parameter set, it was proved that the unweighted dense reaction graph containing the maximum number of reactions corresponding to an uncertain model, forms a superstructure among the possible realizations assuming a fixed complex set. This means that any unweighted reaction graph realizing any kinetic ODE within an uncertain model is a subgraph of the unweighted directed graph of the dense realization.

To search through the possible graph structures, an optimization-based computational model was introduced, where the decision variables are the reaction rate coefficients, and the entries of the monomial coefficient matrix. It was shown that the dense realization can be computed in polynomial time using linear programming steps. An algorithm was proposed to compute those `invariant' reactions (called core reactions) of uncertain models, that are present in any realization of the uncertain model. Most importantly, an algorithm with correctness proof was also proposed in the paper for enumerating all possible reaction graph structures for an uncertain kinetic model.

The theoretical results and proposed algorithms were illustrated on two examples. The examples show that the proposed approach is suitable for the structural uniqueness analysis of uncertain kinetic models.
\section*{Acknowledgements}
This project was developed with the support of the PhD program of the Roska Tam\'as Doctoral School of Sciences and 
Technology, Faculty of Information Technology and Bionics, P\'azm\'any P\'eter Catholic University, Budapest. 
The authors gratefully acknowledge the support of grants PPKE KAP-1.1-16-ITK, and K115694 of the National Research, Development and Innovation Office - NKFIH.


\begin{thebibliography}{10}

\bibitem{Anderson2011}
D.~F. Anderson.
\newblock A proof of the {Global Attractor Conjecture} in the single linkage
  class case.
\newblock {\em {SIAM} Journal on Applied Mathematics}, 71:1487--1508, 2011.
\newblock http://arxiv.org/abs/1101.0761,.

\bibitem{Badri2017}
V.~Badri, M.~J. Yazdanpanah, and M.~S. Tavazoei.
\newblock On stability and trajectory boundedness of {Lotka–Volterra} systems
  with polytopic uncertainty.
\newblock {\em {IEEE} Transactions on Automatic Control}, pages available
  online, DOI 10.1109/TAC.2017.2663839, 2017.

\bibitem{Boyd1994}
S.~Boyd, L.~El-Ghaoui, E.~Feron, and V.~Balakrishnan.
\newblock {\em Linear Matrix Inequalities in Systems and Control Theory}.
\newblock SIAM Books, Philadelphia, PA, 1994.

\bibitem{Briggs2016}
W.~Briggs.
\newblock {\em Uncertainty: The Soul of Modeling, Probability \& Statistics}.
\newblock Springer, 2016.

\bibitem{Chellaboina2009}
V.~Chellaboina, S.~P. Bhat, W.~M. Haddad, and D.~S. Bernstein.
\newblock Modeling and analysis of mass-action kinetics -- nonnegativity,
  realizability, reducibility, and semistability.
\newblock {\em {IEEE} Control Systems Magazine}, 29:60--78, 2009.

\bibitem{Chen2012}
W.~W. Chen, M.~Niepel, and P.~K. Sorger.
\newblock Classic and contemporary approaches to modeling biochemical
  reactions.
\newblock {\em Geners \& Development}, 24:1861--1875, 2012.

\bibitem{Chis2011}
O.~T. Chis, J.~R. Banga, and E.~Balsa-Canto.
\newblock Structural identifiability of systems biology models: a critical
  comparison of methods.
\newblock {\em {PLOS} One}, 6(11):e27755, 2011.

\bibitem{Chis2016}
O.~T. Chis, A.~F. Villaverde, J.~R. Banga, and E.~Balsa-Canto.
\newblock On the relationship between sloppiness and identifiability.
\newblock {\em Mathematical Biosciences}, 282:147--161, 2016.

\bibitem{Craciun2015}
G.~Craciun.
\newblock Toric differential inclusions and a proof of the global attractor
  conjecture.
\newblock arXiv:1501.02860 [math.DS], January 2015.

\bibitem{Craciun2008}
G.~Craciun and C.~Pantea.
\newblock Identifiability of chemical reaction networks.
\newblock {\em Journal of Mathematical Chemistry}, 44:244--259, 2008.

\bibitem{Feinberg:79}
M.~Feinberg.
\newblock {\em Lectures on chemical reaction networks}.
\newblock Notes of lectures given at the Mathematics Research Center,
  University of Wisconsin, 1979.

\bibitem{Feinberg1987}
M.~Feinberg.
\newblock Chemical reaction network structure and the stability of complex
  isothermal reactors - {I}. {T}he deficiency zero and deficiency one theorems.
\newblock {\em Chemical Engineering Science}, 42 (10):2229--2268, 1987.

\bibitem{Guy2015}
T.~V. Guy, M.~Karny, and D.~H. Wolpert, editors.
\newblock {\em Decision Making: Uncertainty, Imperfection, Deliberation and
  Scalability}.
\newblock Springer, 2015.

\bibitem{Haddad2010}
W.~M. Haddad, VS. Chellaboina, and Q.~Hui.
\newblock {\em Nonnegative and Compartmental Dynamical Systems}.
\newblock Princeton University Press, 2010.

\bibitem{Harrison1979}
G.~W. Harrison.
\newblock compartmental models with uncertain flow rates.
\newblock {\em Mathematical Biosciences}, 43:131--139, 1979.

\bibitem{Horn1972}
F.~Horn and R.~Jackson.
\newblock General mass action kinetics.
\newblock {\em Archive for Rational Mechanics and Analysis}, 47:81--116, 1972.

\bibitem{Hars1981}
V.~Hárs and J.~Tóth.
\newblock On the inverse problem of reaction kinetics.
\newblock In M.~Farkas and L.~Hatvani, editors, {\em Qualitative Theory of
  Differential Equations}, volume~30 of {\em Coll. Math. Soc. J. Bolyai}, pages
  363--379. North-Holland, Amsterdam, 1981.

\bibitem{Johnston2011conj}
M.~D. Johnston and D.~Siegel.
\newblock Linear conjugacy of chemical reaction networks.
\newblock {\em Journal of Mathematical Chemistry}, 49:1263--1282, 2011.

\bibitem{Johnston2012a}
M.~D. Johnston, D.~Siegel, and G.~Szederkényi.
\newblock Dynamical equivalence and linear conjugacy of chemical reaction
  networks: new results and methods.
\newblock {\em MATCH Commun. Math. Comput. Chem.}, 68:443--468, 2012.

\bibitem{Johnston2012}
M.~D. Johnston, D.~Siegel, and G.~Szederkényi.
\newblock A linear programming approach to weak reversibility and linear
  conjugacy of chemical reaction networks.
\newblock {\em Journal of Mathematical Chemistry}, 50:274--288, 2012.

\bibitem{Johnston2013}
M.~D. Johnston, D.~Siegel, and G.~Szederkényi.
\newblock Computing weakly reversible linearly conjugate chemical reaction
  networks with minimal deficiency.
\newblock {\em Mathematical Biosciences}, 241:88--98, 2013.

\bibitem{Liebermeister2005}
W.~Liebermeister and E.~Klipp.
\newblock Biochemical networks with uncertain parameters.
\newblock {\em IEE Proceedings Systems Biology}, 152:97--107, 2005.

\bibitem{Liptak2015}
G.~Lipták, G.~Szederkényi, and K.~M. Hangos.
\newblock Computing zero deficiency realizations of kinetic systems.
\newblock {\em Systems \& Control Letters}, 81:24--30, 2015.

\bibitem{Liptak2016}
G.~Lipták, G.~Szederkényi, and K.~M. Hangos.
\newblock Kinetic feedback design for polynomial systems.
\newblock {\em Journal of Process Control}, 41:56--66, 2016.

\bibitem{Ljung1999}
L.~Ljung.
\newblock {\em System Identification - Theory for the User}.
\newblock Prentice Hall, 1999.

\bibitem{Llaneras2007}
F.~Llaneras and J.~Picó.
\newblock An interval approach for dealing with flux distributions and
  elementary modes activity patterns.
\newblock {\em Journal of Theoretical Biology}, 246:290--308, 2007.

\bibitem{Lodish2000}
H.~F. Lodish.
\newblock {\em Molecular cell biology}.
\newblock W.H. Freeman, New York, 2000.

\bibitem{Nagy2011}
T.~Nagy and T.~Turányi.
\newblock Uncertainty of {Arrhenius} parameters.
\newblock {\em International Journal of Chemical Kinetics}, 43:359--378, 2011.

\bibitem{Schillings2015}
C.~Schillings, M.~Sunnaker, J.~Stelling, and C.~Schwab.
\newblock Efficient characterization of parametric uncertainty of complex
  biochemical networks.
\newblock {\em {PLOS} Computational Biology}, 11(8):e1004457, 2015.

\bibitem{Schnell2006}
S.~Schnell, M.~J. Chappell, N.~D. Evans, and M.~R. Roussel.
\newblock The mechanism distinguishability problem in biochemical kinetics: The
  single-enzyme, single-substrate reaction as a case study.
\newblock {\em Comptes Rendus Biologies}, 329:51--61, 2006.

\bibitem{Shinar2010}
G.~Shinar and M.~Feinberg.
\newblock Structural sources of robustness in biochemical reaction networks.
\newblock {\em Science}, 327:1389--1391, 2010.

\bibitem{Son:2001}
E.~Sontag.
\newblock Structure and stability of certain chemical networks and applications
  to the kinetic proofreading model of {T}-cell receptor signal transduction.
\newblock {\em {IEEE} Transactions on Automatic Control}, 46:1028--1047, 2001.

\bibitem{Szederkenyi2009b}
G.~Szederkényi.
\newblock Computing sparse and dense realizations of reaction kinetic systems.
\newblock {\em Journal of Mathematical Chemistry}, 47:551--568, 2010.

\bibitem{Szederkenyi2011c}
G.~Szederkényi, J.~R. Banga, and A.~A. Alonso.
\newblock Inference of complex biological networks: distinguishability issues
  and optimization-based solutions.
\newblock {\em BMC Systems Biology}, 5:177, 2011.

\bibitem{Szederkenyi2011a}
G.~Szederkényi and K.~M. Hangos.
\newblock Finding complex balanced and detailed balanced realizations of
  chemical reaction networks.
\newblock {\em Journal of Mathematical Chemistry}, 49:1163--1179, 2011.

\bibitem{Szederkenyi2011}
G.~Szederkényi, K.~M. Hangos, and T.~Péni.
\newblock Maximal and minimal realizations of reaction kinetic systems:
  computation and properties.
\newblock {\em MATCH Commun. Math. Comput. Chem.}, 65:309--332, 2011.

\bibitem{Szederkenyi2016a}
G.~Szederkényi, B.~Ács, and G.~Szlobodnyik.
\newblock Structural analysis of kinetic systems with uncertain parameters.
\newblock In {\em 2nd IFAC Workshop on Thermodynamic Foundations for a
  Mathematical Systems Theory - TFMST, {IFAC} Papersonline}, volume~49, pages
  24--27, 2016.

\bibitem{Tak:96}
Y.~Takeuchi.
\newblock {\em Global Dynamical Properties of Lotka-{V}olterra Systems}.
\newblock World Scientific, Singapore, 1996.

\bibitem{Tuza2015}
Z.~A. Tuza and G.~Szederkényi.
\newblock Computing core reactions of uncertain polynomial kinetic systems.
\newblock In {\em 23rd Mediterranean Conference on Control and Automation
  (MED), June 16-19, 2015. Torremolinos, Spain}, pages 1187--1194, 2015.

\bibitem{Weinmann1991}
A.~Weinmann.
\newblock {\em Uncertain Models and Robust Control}.
\newblock Springer-Verlag, 1991.

\bibitem{Wu2006}
J-L. Wu.
\newblock Robust stabilization for single-input polytopic nonlinear systems.
\newblock {\em {IEEE} Transactions on Automatic Control}, 2006.

\bibitem{Yi2003}
T.~M. Yi, H.~Kitano, and M.~I. Simon.
\newblock A quantitative characterization of the yeast heterotrimeric g protein
  cycle.
\newblock {\em Proc. Natl. Acad. Sci. USA}, 100(19):10764–10769, 2003.

\bibitem{Acs2017}
B.~Ács, G.~Szederkényi, and D.~Csercsik.
\newblock A new efficient algorithm for determining all structurally different
  realizations of kinetic systems.
\newblock {\em MATCH Commun. Math. Comput. Chem.}, 77:299--320, 2017.

\bibitem{Acs2015}
B.~Ács, G.~Szederkényi, Z.~A. Tuza, and Z.~Tuza.
\newblock Computing linearly conjugate weakly reversible kinetic structures
  using optimization and graph theory.
\newblock {\em MATCH Commun. Math. Comput. Chem.}, 74:481--504, 2015.

\bibitem{Acs2016}
B.~Ács, G.~Szederkényi, Zs. Tuza, and Z.~A. Tuza.
\newblock Computing all possible graph structures describing linearly conjugate
  realizations of kinetic systems.
\newblock {\em Computer Physics Communications}, 204:11--20, 2016.

\bibitem{Erdi1989}
P.~Érdi and J.~Tóth.
\newblock {\em Mathematical Models of Chemical Reactions. Theory and
  Applications of Deterministic and Stochastic Models}.
\newblock Manchester University Press, Princeton University Press, Manchester,
  Princeton, 1989.

\end{thebibliography}
\end{document}